\RequirePackage{fix-cm}
\documentclass[smallextended]{svjour3}       % onecolumn (second format)
\smartqed  % flush right qed marks, e.g. at end of proof
\usepackage{graphicx}
\usepackage{amsmath, ntheorem}
\usepackage{amssymb}
\usepackage{stackrel}
\usepackage{marvosym, enumerate}
\usepackage{amstext, amscd, latexsym}
\usepackage{amsfonts, ragged2e}
\usepackage{mathrsfs, pgfplots, enumerate, setspace}
\usepackage{subfigure}
\usepackage{cases}
\usepackage{booktabs}
\usepackage{tabularx}
\usepackage[ruled,linesnumbered]{algorithm2e}
\usepackage{float} %fix the table and figure
\usepackage[font={bf},textfont=md,labelfont={footnotesize},labelformat={default}]{caption}
\numberwithin{equation}{section}
\smartqed
\usepackage{cite}
\usepackage[colorlinks,linkcolor=blue,anchorcolor=blue,citecolor=blue]{hyperref}% 超链接的颜色
\usepackage{amstext, graphicx, amscd, latexsym, amssymb}
\usepackage{amsfonts, ragged2e}
\usepackage{pgfplots, setspace}
\usepackage{geometry}
\usepackage{latexsym,bm}
\usepackage{float}
\usepackage{threeparttable}
\usepackage{array}
\usepackage{booktabs}
\usepackage{makecell}
\usepackage[justification=centering,labelfont=bf]{caption}
\usepackage{etoolbox}
\usepackage[figuresright]{rotating}
\begin{document}

\title{Reference-Point-Based Branch and Bound Algorithm for Multiobjective Optimization}

\titlerunning{}        % if too long for running head

\author{Weitian Wu$^1$ \and Xinmin Yang$^2$ }

%\authorrunning{} % if too long for running head

\institute{ W.T. Wu \at College of Mathematics, Sichuan University, Chengdu 610065, China\\
                    \email{wuweitian@stu.scu.edu.cn}\\
           \Letter X.M. Yang \at National Center for Applied Mathematics of Chongqing 401331, China\\
           School of Mathematical Sciences,  Chongqing Normal University, Chongqing 401331, China\\
              xmyang@cqnu.edu.cn}

\date{Received: date / Accepted: date}

\maketitle

\begin{abstract}
In this paper, a branch and bound algorithm that incorporates the decision maker's preference information is proposed for multiobjective optimization. In the proposed algorithm, a new discarding test is designed to check whether a box contains preferred solutions according to the preference information expressed by means of reference points. In this way, the proposed algorithm is able to gradually guide the search towards the region of interest on the Pareto fronts during the solution process. We prove that the proposed algorithm obtains $\varepsilon$-efficient solutions distributed in the region of interest. Moreover, lower bound on the total finite number of required iterations for predefined precision is also provided. Finally, the algorithm is illustrated with a number of test problems.

\keywords{Multiobjective optimization \and Branch and bound algorithm \and Preference information\and Reference point}
\subclass{90C26\and 90C29\and 90C30}
\end{abstract}

\section{Introduction}
Many real-world optimization problems require to take into account several conflicting objectives simultaneously. Researchers formulate such problems as multiobjective optimization problems (MOPs). Since objectives are conflicting, it is impossible to find a unique solution to optimize each objective to its fullest. Instead, a number of Pareto solutions can be identified, which are characterized by the fact that an improvement in any one objective can only be achieved at the expense of a degradation in at least one other objective. Therefore, without the intervention of the decision maker, none of the Pareto solutions can be said to be worse than others. Following a classification by Hwang and Masud \cite{ref8}, the participation of the decision maker may be done either before (a priori), during (interactive), or after (a posteriori) the solution process.

The a priori methods demand the decision maker to specify a reference point or a reference direction or other preference information before solution process. In these methods, based on such information, a single objective optimization problem can be formed and a single solution be found. However, they are usually not practicable, since it is difficult for the decision maker to explicitly and exactly quantify his/her preferences before any alternatives are known, so the single solution found may not be the preferred one.

Recently, the a posteriori methods have received a great deal of attention from academia \cite{ref5,ref9} and industry \cite{ref2,ref10}. They attempt to provide the decision maker with an entirely approximation of the Pareto set. In this way, the decision maker can look at a large set of generated alternatives before making a decision and thereby revealing her/his explicitly and exactly preferences. As a result, there is no need for re-optimization or further interaction with the decision maker, and the decision maker has greater confidence in the final decision.%Hence no re-optimizations and further interaction are needed for the decision maker, which reinforces his/her confidence on the final decision.

Branch and bound algorithms for MOPs \cite{ref6,ref7,ref17,ref23,ref25,ref26} can be classified as a posteriori. They obtain a covering of the entire Pareto set by successively subdividing the variable space into smaller regions and pruning off regions that are provably suboptimal. The reason for their success is the simplicity and extensibility of the method. As a result, many enhancements to branch and bound algorithms for MOPs have been developed, leading to dramatic improvements in solution process. For instance, Fern{\'a}ndez and T{\'o}th \cite{ref7} design three discarding tests for their branch and bound algorithm based on the monotonicity of the objective functions, which speed up the convergence of the algorithm remarkably. The branch and bound algorithms proposed by {\v{Z}}ilinskas et al. \cite{ref25,ref26} use the Lipschitz constants of the objectives to construct lower bounds for the Pareto front. The branch-and-bound-based algorithm presented by Niebling and Eichfelder \cite{ref8} adopts a new discarding test that combines the $\alpha$BB method \cite{ref1} with an extension of Benson's outer approximation techniques \cite{ref5}. Wu and Yang \cite{ref23} propose a parallel branch-and-bound-based framework which employs heuristic search to improve the tightness of the lower and upper bounds.

While a set of alternatives provided by branch and bound algorithms may be very reliable for the decision maker, the search for the whole Pareto set poses high computation cost, if particular the dimension of the variable space is large. The computation cost mainly consists of searching for bounds, storing the exponential number of subboxes, and performing the discarding test. However, as stated for example by Deb in \cite{ref4}, the decision maker is usually not just looking for a single solution, rather she/he is interested in knowing the properties of solutions which are in the region of interest respecting the preference information. In other words, if a number of solutions distributed in the region of interest are found, the decision maker is still able to make a reliable decision. Therefore, in our view, how to use the preferences to avoid exploring undesired regions is the key to reducing the computational cost of branch and bound algorithms.

In this paper, we propose a new interactive algorithm whose principle is to incorporate preference information coming from a decision maker into the branch and bound algorithm. As mentioned earlier, we are interested in approximating a part of the Pareto set instead of the whole one. Hence the decision maker is asked to provide preference information in terms of his/her reference point before solution process, which consists of desirable aspiration levels for objectives. The reference point is usually used together with an  achievement scalarizing function which is integrated into a new discarding test. With the help of the new discarding test, the proposed algorithm is able to gradually guide the search towards the region of interest by excluding the subboxes that do not contain the preferred solution from the exploration. Furthermore, if the decision maker provides multiple reference points, the algorithm can bias the search towards different regions of interest simultaneously. In addition, a heuristic method is used in order to improve the solution quality, and we prove that the proposed algorithm computes a set of $\varepsilon$-efficient solutions.

The rest of this paper is organized as follows. In Section 2, we introduce the basic concepts and notations of multiobjective optimization. The \emph{Reference-point-based Branch and Bound algorithm} (RBB) with a new discarding test is described in Section 3. Section 4 is devoted to some theoretical results. Numerical results are presented in Section 5.

\section{Basics of multiobjective optimization}
In this section we introduce the basic concepts which we need for the new algorithm. A multiobjective optimization problem can be written as follows:
\begin{align}\label{MOP}
\min\limits_{x\in\Omega}\quad F(x)=(f_1(x),\ldots ,f_m(x))^T
\end{align}
with
\begin{align*}
\Omega=\{x\in\mathbb{R}^n:g_j(x)\geq0,\;j=0,\ldots,p,\;\underline{x}_k\leq x_k\leq \overline{x}_k,\;k=0,\ldots,n\},
\end{align*}
where $f_i:\mathbb{R}^n\rightarrow \mathbb{R}$ ($i=1,\ldots,m$) are Lipschitz continuous, and $g_j:\mathbb{R}^n\rightarrow \mathbb{R}$ ($j=0,\ldots,p$) are continuous. If we allow $j=0$, the set $\Omega$ is referred to as a box constraint. In this case, we call $\Omega$ a \emph{box} with the midpoint $m(\Omega)=(\frac{\underline{x}_1+\overline{x}_1}{2},\ldots,\frac{\underline{x}_n+\overline{x}_n}{2})^T$ and the width $\omega(\Omega)=\|\overline{x}-\underline{x}\|$, where $\|\cdot\|$ denotes the Euclidean norm. For a feasible point $x\in\Omega$, the objective vector $F(x)\in\mathbb{R}^m$ is said to be the image of $x$, while $x$ is called the preimage of $F(x)$.

The concept of Pareto dominance relation is used to compare any two points $x,y\in\Omega$ and is defined by
\begin{align*}
x\;weakly\;dominates\;y\;&\Longleftrightarrow\;F(x)\leqq F(y)\;\Longleftrightarrow\;F(y)-F(x)\in \mathbb{R}^m_+,\\
x\;dominates\;y\;\qquad&\Longleftrightarrow\;F(x)\leq F(y)\;\Longleftrightarrow\;F(y)-F(x)\in \mathbb{R}^m_+\backslash\{0\},\\
x\;strictly\;dominates\;y&\Longleftrightarrow\;F(x)< F(y)\;\Longleftrightarrow\;F(y)-F(x)\in {\rm int}\mathbb{R}^m_+.
\end{align*}
where $\mathbb{R}^m_+=\{z\in\mathbb{R}^m:z\geq0\}$ and \emph{int} denotes the interior. When $F(x)\nleq F(y)$, $F(y)\nleq F(x)$ and $F(y)\neq F(x)$, we say that $x$ is \emph{indifferent} to $y$ ($x\sim y$). We can define similar terms in the objective space. A set $N\subseteq\mathbb{R}^m$ is called a \emph{nondominated set} if for any $z^1,z^2\in N$ we have $z^1\sim z^2$.

A point $x^*\in\Omega$ is said to be a \emph{Pareto solution} for (MOP) if there does not exist any $x\in\Omega$ such that $F(x)\leq F(x^*)$. The set of all Pareto solutions is called the \emph{Pareto set} and is denoted by $X^*$. The image of Pareto set under the mapping $F$ is called the \emph{Pareto front}.

The aim of an approximation algorithm is used to found an $\varepsilon$-efficient solution of problem (\ref{MOP}), which is defined next. Let $e$ denote the $m$-dimensional all-ones vector $(1,\ldots,1)^T\in\mathbb{R}^m$.

\begin{definition}\cite{ref13}\label{eep}
Let $\varepsilon\geq0$ be given A point $\bar{x}\in \Omega$ is an $\varepsilon$-efficient solution of problem (\ref{MOP}) if there does not exists another $x\in \Omega$ with $F(x)\leq F(\bar{x})-\varepsilon e$.
\end{definition}

In order to obtain a single solution preferred by the decision maker, all Pareto solutions must be put in a complete order. This is why we need the involvement of preference information. A common way of expressing preference information is to specify aspiration levels for all objective function values, which constitute the components of the \emph{reference point}. In general, a reference point is said to be \emph{feasible} if its preimage is a feasible point in variable space; otherwise, the reference point is said to be \emph{infeasible}.

The reference point is usually used in \emph{achievement scalarizing functions} (ASFs) which can convert problem (\ref{MOP}) into a single-objective optimization problem with reference points. The solution of the single-objective optimization problem is referred to as \emph{the most preferred solution} of problem (\ref{MOP}).

One of the most used ASF called \emph{augmented weighted achievement function} is proposed by Wierzbicki\cite{ref22}, which is related to augmented weighted Tchebycheff metric. For a given reference point $r\in \mathbb{R}^m$, a weight vector $w\in \mathbb{R}^m_{++}$ and an augmentation coefficient $\rho>0$, the augmented weighted achievement function is given by
\begin{align}
s(x,r,w)=\max\limits_{i=1,\ldots,m}w_i(f_i(x)-r_i)+\rho\sum_{i=1}^{m}w_i(f_i(x)-r_i)\label{ASF1}
\end{align}
which must be minimized over $\Omega$:
\begin{align}
\qquad\min\limits_{x\in\Omega}\quad s(x,r,w)\label{sxr}
\end{align}
The coefficient $\rho>0$ must be a small positive value (i.e. $\rho=10^{-6}$) which ensures that the most preferred solution is a (properly) Pareto solution of problem (\ref{MOP}) for any reference point in $\mathbb{R}^m$ (see e.g.\cite{ref14}). Noting that, when the ASF given in \eqref{ASF1} is minimized over $\Omega$ using a strictly positive weight, in practice, the reference point is projected onto the Pareto front in the direction defined by the inverses of the weight used \cite{ref19,ref14}. Therefore, an augmented weighted achievement function for an objective vector $z\in\mathbb{R}^m$ can be defined as follows:
\begin{align}
\bar{s}(z,r,w)=\max\limits_{i=1,\ldots,m}\frac{1}{w_i}(z_i-r_i)+\rho\sum_{i=1}^{m}\frac{1}{w_i}(z_i-r_i).\label{ASF2}
\end{align}

We use $d(a,b)=\|a-b\|$ to quantify the distance between two points $a$ and $b$, and the distance between the point $a$ and a non-empty finite set $B$ is defined as $d(a,B):=\min_{b\in B}\|a-b\|.$ Let $A$ be another non-empty finite set, we define the directed Hausdorff distance from $A$ to $B$ by
\begin{align*}
d_h(A,B):=\max_{a\in A}\{\min_{b\in B}\|a-b\|\}.
\end{align*}

%For an introduction to the basic branch and bound algorithm for multiobjective optimization we refer the reader to \cite{FT}.

\section{Reference-point-based branch and bound algorithm}

In the new algorithm to be introduced, we treat the branch and bound method as the framework. For this reason, we describe the solution process of a basic branch and bound algorithm for MOPs in more detail.
\subsection{New discarding test}
A branch and bound algorithm systematically searches for an approximation of the whole Pareto set by means of a tree search in the variable space. By pruning off nodes (subboxes) in the tree that are provably suboptimal, it is possible to limit the tree search and thus avoid exhaustive enumeration. At each iteration, the algorithm bisects a box perpendicularly to a direction of maximum width, resulting in two mutually exclusive subboxes (\emph{branching}). Each subbox corresponds to a subproblem of problem (\ref{MOP}), i.e., minimizing the objective function $F$ over one subbox. The upper and lower bounds for the Pareto front corresponding to one subbox are calculated to check if the subbox contains any Pareto solution of problem (\ref{MOP}). If not, the subbox is removed, which is referred as to the \emph{pruning}, while the calculation of the lower and upper bounds is known as \emph{bounding}. Given a box $B$, the source of the upper bounds $u$ with respect to $B$ is the image of the midpoint or the vertexes of $B$. Each component of the lower bound $l=(l_1,\ldots,l_m)^T$ with respect to $B$ can be calculated as follows \cite{ref12}:
\begin{align}
l_i = f_i(m(B))-\frac{L_i}{2}\omega(B), i=1,\ldots,m,\label{lb}
\end{align}
where $L_i$ is the Lipschitz constant of $f_i$. The pruning can be achieved by a \emph{discarding test}. A common type of discarding test is based on the Pareto dominance relation:\\
\noindent{\bf Preference-free discarding test} Let problem (\ref{MOP}) be given, let a subbox $B\in\Omega$ with its lower bound $l(B)\in\mathbb{R}^m$. Further, let $\mathcal{U}^{nds}$ be the nondominated upper bound set of problem (\ref{MOP}). Then, the box $B$ will be discarded if there exists $u\in \mathcal{U}^{nds}$ such that $u\leq l$.

A common point in many branch and bound algorithms \cite{ref6,ref7,ref17,ref23,ref25,ref26} is the absence of preference information in the solution process. As mentioned above, branch and bound algorithms try to generate an entirely approximation of the Pareto set assuming that any nondominated solution is desirable. But this is not always the case in a real situation where different parts of the Pareto set could be more preferred than some others, and some parts could not be interesting at all. From our point of view, this lack of preference information produces shortcomings in two ways:

\begin{itemize}
  \item Computational resources are wasted in exploring undesired subboxes; and
  \item The decision maker may be unable to find the most preferred solution among a huge number of alternatives when the problem has more than three objectives, because the visualization of the Pareto fronts for many-objective problems is not as illustrative or intuitive as for two objectives.
\end{itemize}

In order to avoid the above-mentioned shortcomings, preference information must be used in the solution process of branch and bound algorithms. In this way, the algorithms can avoid exploring undesired subboxes, and decision maker uses the preference information to guide the search towards the preferred solution.

Here we integrate preference information given by the decision maker in the form of reference points into the branch and bound algorithm such that subboxes gradually concentrate in the neighborhood of those solutions that obey the preference as well as possible. As noticed in Ruiz et al. \cite{ref19}, the reference points and achievement scalarizing functions are fundamentally related to each other. Hence we use the ASF given in \eqref{ASF2} embedded into the preference-free discarding test. In the following, we give the preference-based discarding test.

\noindent{\bf Preference-based discarding test} Let problem (\ref{MOP}) be given, let a subbox $B\in\Omega$ with its lower bound $l(B)\in\mathbb{R}^m$, and let $\mathcal{U}^{p}$ be a preferred upper bound set of problem (\ref{MOP}). Further, let a reference point $r\in\mathbb{R}^m$, a weight vector $w\in\mathbb{R}^m_{++}$ and an algorithmic parameter $\sigma>0$ be given. Then, the box $B$ will be discarded if one of the following discarding conditions holds true:
\begin{enumerate}[a)]
  \item there exists a point $u\in \mathcal{U}^{p}$ such that $\bar{s}(u,r,w)\leq\bar{s}(l(B),r,w)-\sigma$;\label{cond1}
  \item there exists a point $u\in \mathcal{U}^{p}$ such that $u\leq l(B)$.\label{cond2}
\end{enumerate}

Now we state the correctness of the preference-based discarding test.

\begin{lemma}\label{lemma1}
Let a subbox $B\in\Omega$ and its lower bound $l(B)\in\mathbb{R}^m$ be given, let $\mathcal{U}^{p}$ be a preferred upper bound set of problem (\ref{MOP}) with respect to a given reference point $r\in\mathbb{R}^m$ and weight vector $w\in\mathbb{R}^m$.
\begin{enumerate}[\rm 1)]
  \item If the discarding condition \ref{cond1}) is satisfied, then $B$ does not contain the most preferred solution of problem (\ref{MOP});
  \item If the discarding condition \ref{cond2}) is satisfied, then $B$ does not contain any Pareto solution of problem (\ref{MOP});
  \item The box containing the most preferred solution will never be discarded.
\end{enumerate}
\end{lemma}
\begin{proof}
Because $l(B)$ is the lower bound with respect to $B$, it follows that $l(B)\leq F(x)$ for all $x\in B$. Then for all $x\in B$, we have $l(B)_i-r_i\leq F(x)_i-r_i$ for all $i=1,\ldots,m$ with the strict inequality holding for at least one index $j$, and further $\bar{s}(l(B),r,w)< s(x,r,w)$. If condition \ref{cond1}) is satisfied, then there exists a upper bound $u\in \mathcal{U}^{p}$ such that $\bar{s}(u,r,w)<\bar{s}(l(B),r,w)-\sigma<\bar{s}(l(B),r,w)$. Thus we have $\bar{s}(u,r,w)< s(x,r,w)$ for all $x\in B$. The conclusion 2) is trivial due to the no-preference discarding test. The conclusion 3) holds because the most preferred solution is a Pareto solution of problem (\ref{MOP}).\qed
\end{proof}

Lemma \ref{lemma1} shows that the preference-based discarding test has ability to remove more boxes that do not contain the most preferred solution by using the ASF. And further, the discarding pressure can be controlled by the predefined parameter $\sigma$, i.e. a larger $\sigma$ indicates that more subboxes are retained, leading to a wider region of interest.

In fact, the introduction of the parameter $\sigma$ is inspired by \emph{r-dominance} \cite{ref20}. Although both of approaches use the parameter $\sigma$ to control the selection (discarding) pressure in order to obtain the region of interest, their main difference lies in the definition of region of interest. In r-dominance, the region of interest is defined by the objective vectors according to their Euclidean distance to the reference point. Whereas the region of interest obtained in the new discarding test is constructed by the Pareto solutions near the most preferred solution. Furthermore, if we replace the ASF in the discarding condition a) with the Euclidean distance, then the resulting discarding test is unable to handle the feasible reference point.

Sometimes, the decision maker will provide multiple reference points in order to obtain several different regions. Therefore, we extend the above discarding test to a version which can handle multiple reference points simultaneously.

\noindent{\bf Discarding test for multiple reference points} Let problem (\ref{MOP}) be given, let a subbox $B\in\Omega$ with its lower bound $l(B)\in\mathbb{R}^m$. Let $\mathcal{U}^{p}$ be a preferred upper bound set of problem (\ref{MOP}), $\mathcal{P}=(r,w)$ a list of preference information, and $\sigma$ a predefined parameter. The box $B$ can be discarded if one of the following discarding conditions holds true:
\begin{enumerate}[a)]
  \item for each element $(r,w)\in\mathcal{P}$, there exists a point $u\in \mathcal{U}^{p}$ such that $\bar{s}(u,r,w)\leq\bar{s}(l(B),r,w)-\sigma$;\label{conda}
  \item there exists a point $u\in \mathcal{U}^{p}$ such that $u\leq l(B)$.\label{condb}
\end{enumerate}

Algorithm 1 gives an implementation of the second discarding test, where the flag $\mathcal{D}$ stands for decision to discard the box after the algorithm. Before performing the discarding test for the box $B$, a set of preferred solutions $\mathcal{U}^p$ should be given, which can be obtained by means of calculating the upper bound set for the preferred boxes at the previous iteration. In the for-loop from line 2, the decision to remove $B$ with respect to every preference in $\mathcal{P}$ is stored in the list $D$. The case where there is a flag equal to 0 in $D$ implies that the box should not be discarded under $\mathcal{P}$.

\begin{algorithm}[H]\label{alg1}
\caption{\texttt{DT}$(B, l(B),\mathcal{U}^{p},\mathcal{P})$}
  \SetKwInOut{Input}{Input}\SetKwInOut{Output}{Output}
  \Input{A subbox $B$, the lower bound $l(B)$, a preferred upper bound set $\mathcal{U}^{p}$, a list of preference information $\mathcal{P}$, a predefined parameter $\sigma$;}
  \Output{The flag $\mathcal{D}$;}
  $D\leftarrow \emptyset$, $\mathcal{D}\leftarrow0$\;
  %Calculate for $B$ its lower bound $l$ by (\ref{lb})\;
  \ForEach{$(r,w)\in\mathcal{P}$}{
  %Extract preference information $(r,w)$ from $P$\;
  \lIf{there exists $u\in\mathcal{U}^{p}$ such that $s_2(u,r,w)< s_2(l(B),r,w)-\sigma$}{$D\leftarrow D\cup \{1\}$\;}%if
  \lElse{$D\leftarrow D\cup \{0\}$\;}%else
  }% for loop
  \If{all flags in $D$ are equal to \rm{1} {\bf or} there exists $u\in\mathcal{U}^{p}$ such that $u\leq l(B)$}{$\mathcal{D}\leftarrow1$.}
\end{algorithm}

\subsection{The complete algorithm}
Having now the new discarding test, we can present the reference-point-based branch and bound algorithm (RBB). As previously mentioned, the main purpose of RBB is to approximate the region(s) of interest of the Pareto front according to the reference point(s) provided by the decision maker. For two given precision parameters $\varepsilon$ and $\delta$, RBB is able to generate a set of $\varepsilon$-efficient solutions distributed on the region(s) of interest.

The pseudocode of RBB is given in Algorithm 2. Considering that Pareto solutions can be distributed in arbitrary boxes (nodes) at the same depth of the branch and bound tree, \emph{the breadth first search strategy} is more suitable for RBB than other strategies, since the breadth first search strategy has the advantage of always finding an optimal solution that is closest to the root of the tree. In order to shorten the computation time, in line 4, RBB constructs the new box collection $\mathcal{B}_{k}$ by simultaneously bisecting all the boxes stored in $\mathcal{B}_{k-1}$ to ensure that the subsequent breadth first search can be parallelized. In addition, many boxes in $\mathcal{B}_{k}$ which do not contain any feasible points should be filtered out. RBB employs the feasibility test suggested in \cite{ref7}.

The next steps of RBB consist of two phases. Phase one (from line 7 to line 18) aims to roughly identify the preferred subboxes that may contain $\varepsilon$-efficient solutions by using the new discarding test; phase two (from line 19 to line 33) aims to improve the solution quality of the preferred boxes.% and remove more non-preferred subboxes.

\begin{algorithm}\label{RBB}
\caption{\texttt{Reference-Point-based Branch and Bound Algorithm}}
  \SetKwInOut{Input}{Input}\SetKwInOut{Output}{Output}
  \SetKwFunction{DT}{DT}
  \SetKwFunction{MOEA}{MOEA}
  \Input{problem (\ref{MOP}), a list of preference information $\mathcal{P}$, $\varepsilon>0$, $\delta>0$;}
    \Output{$\mathcal{B}_{k}$, $\mathcal{L}^{p}$, $\mathcal{U}^{p}$, $\mathcal{X}$;}
    $k\leftarrow1$, $\mathcal{B}_0\leftarrow\Omega$, $\omega_{k-1}\leftarrow \omega(\Omega)$, $d\leftarrow10^6$\;
  \While{$d>\varepsilon$ or $\omega_{k-1}>\delta$}{
  $\mathcal{I}\leftarrow\emptyset$, $\mathcal{L}^{p}\leftarrow\emptyset$, $\mathcal{U}^{p}\leftarrow \emptyset$, $\mathcal{X}^{p}\leftarrow \emptyset$\;
  Construct $\mathcal{B}_{k}$ by bisecting all boxes in $\mathcal{B}_{k-1}$\;
  $\omega_k\leftarrow \max\{\omega(B):B\in\mathcal{B}_{k}\}$\;
  Update $\mathcal{B}_{k}$ by the feasibility test suggested in \cite{ref7}\;
  \ForEach{$B\in\mathcal{B}_{k}$}{
  Treat the image of the midpoint of $B$ as the upper bound $u(B)$\;
  Calculate for $B$ its lower bound $l(B)$ by (\ref{lb})\;
  $\mathcal{I}\leftarrow\mathcal{I}\cup\{(l(B);u(B);B)\}$\;}
  Extract the lower bound set $\mathcal{L}$ from $\mathcal{I}$\;
  Find a nondominated lower bound set $\mathcal{L}^{p}$ from $\mathcal{L}$\;
  Update $\mathcal{U}^p$ and $\mathcal{B}^p$ according to $\mathcal{L}^{p}$\;
  \ForEach{$B\in\mathcal{B}_{k}$}{
  $\mathcal{D}\leftarrow \DT(B,l(B),\mathcal{U}^{p},\mathcal{P})$\;
  \lIf{$\mathcal{D}=1$}
  {$\mathcal{B}_{k}\leftarrow \mathcal{B}_{k}\backslash B$\;}
  }
  \ForEach{$B\in\mathcal{B}^p$}
  {Obtain an upper bound set $\mathcal{U}$ and a solution set $\mathcal{X}$ by applying \MOEA to $B$\;
  Update $\mathcal{U}^p$ by $\mathcal{U}$ and $\mathcal{X}^p$ by $\mathcal{X}$\;}

  \If{$|\mathcal{U}^p|<M$}{  \ForEach{$B\in\mathcal{B}_k$}
  {Obtain an upper bound set $\mathcal{U}$ and a solution set $\mathcal{X}$ by applying \MOEA to $B$\;
  Update $\mathcal{U}^p$ by $\mathcal{U}$ and $\mathcal{X}^p$ by $\mathcal{X}$\;}}

  \ForEach{$B\in\mathcal{B}_{k}$}{
  $\mathcal{D}\leftarrow \DT(B,l(B),\mathcal{U}^{p},\mathcal{P})$\;
  \lIf{$\mathcal{D}=1$}
  {$\mathcal{B}_{k}\leftarrow \mathcal{B}_{k}\backslash B$\;}
  }

  $d\leftarrow d_h(\mathcal{U}^{p},\mathcal{L}^{p})$, $k\leftarrow k+1$.
  }
\end{algorithm}

In phase one, the first for-loop from line 7 calculates for each box $B\in\mathcal{B}_k$ its lower bound $l(B)$ and upper bound $u(B)$, where $l(B)$ can be calculated by \eqref{lb} and $u(B)$ is the image of the midpoint of $B$. Then, the pair $(l(B);u(B);B)$ is stored in the information list $\mathcal{I}$ in order to track the data. From line 12 to line 14, a nondominated lower bound set $\mathcal{L}^p$ is selected from the lower bound set $\mathcal{L}$ which are stored in $\mathcal{I}$; and then, for each lower bound in $\mathcal{L}^p$, its corresponding upper bound and box are inserted into the preferred upper bound set $\mathcal{U}^p$ and the preferred box collection $\mathcal{B}^p$, respectively. The second for-loop of RBB checks whether the box in $\mathcal{B}_k$ can be removed by the preference-based discarding test.

In phase two, an MOEA with a small initial population and a few generations is applied to each subbox in $\mathcal{B}^p$ in order to obtain the upper bound set $\mathcal{U}$ and the solution set $\mathcal{X}$. Then the upper bound set $\mathcal{U}$ is used to update the preferred upper bound set $\mathcal{U}^p$, i.e. for each upper bound $u\in\mathcal{U}$ we check if $u$ is dominated by any other upper bound in $\mathcal{U}^p$. In this case $u$ is not included in $\mathcal{U}^p$; otherwise, $u$ is added to $\mathcal{U}^p$ and all upper bounds dominated by $u$ are removed. the preferred solution set $\mathcal{X}^p$ changes according to $\mathcal{U}^p$. Furthermore, a constraint handling technique \cite{ref11} is used in the MOEA to improve the number of feasible upper bounds. If the number of upper bounds in $\mathcal{U}^p$ is less than $M$ ($M$ is the minimum number of candidate solutions required by the decision maker, here we set $M=100$), the MOEA is applied to each subbox in $\mathcal{B}_k$ in order to obtain more feasible upper bounds.

Additionally, we incorporate an objective normalization technique into the achievement scalarizing function to ensure that the resulting region of interest correctly expresses the trade-off among disparately scaled objectives. The ASF given in \eqref{ASF2} is replaced by
\begin{align*}
\max\limits_{i=1,\ldots,m}\frac{1}{w_i}\Big(\frac{z_i-r_i}{\tilde{z}^{nad}_i-\tilde{z}^*_i}\Big)+\rho\sum_{i=1}^{m}\frac{1}{w_i}\Big(\frac{z_i-r_i}{\tilde{z}^{nad}_i-\tilde{z}^*_i}\Big),
\end{align*}
where $\tilde{z}^{nad}_i$ are the largest value of $f_i$ in the current upper bounds, and $\tilde{z}^*_i$ are the smallest value of $f_i$ in the current lower bounds.

The initial $\tilde{z}^{nad}$ and $\tilde{z}^*$ are estimated by the natural interval extension \cite{ref16}. In order to update $\tilde{z}^*$, at each iteration the subbox with the smallest value of $f_i$ in the current lower bounds remains in the box collection, and in line 13 each component of $\tilde{z}^*$ is updated by the smallest value of $f_i$ in the current lower bounds. At the same time, each component of $\tilde{z}^{nad}$ is updated by the largest value of $f_i$ in the current upper bounds. The update of $\tilde{z}^{nad}$ can be done in lines 14, 21 and 26.

\section{Convergence results}
First, we state that the most preferred solution of problem (\ref{MOP}) is always contained in the box collection generated by RBB.
\begin{lemma}
Let $\{\mathcal{B}_k\}_{k\in\mathbb{N}}$ be a sequence of box collections generated by RBB. Then, for the most preferred solution $x^p$ of problem (\ref{MOP}), we have
\begin{align*}
x^p\subset\cdots\subset \mathcal{B}_k\subset\cdots\subset\mathcal{B}_1 \subset \mathcal{B}_0.
\end{align*}
\end{lemma}
\begin{proof}
This conclusion is guaranteed by Lemma \ref{lemma1} and the way $\mathcal{B}_k$ is constructed.\qed
\end{proof}

Next we have to verify that RBB is finite.
\begin{theorem}
Let the predefined parameters $\varepsilon>0$ and $\delta>0$ be given, RBB terminates.
\end{theorem}
\begin{proof}
Because we divide all boxes perpendicular to a side with maximal width, $\omega_k$ decreases among the sequence of box collections, i.e., $\omega_{k} > \omega_{k+1}$ for every $k$ and converges to 0. Therefore, for a given $\delta>0$, there must exist a iteration count $\tilde{k}>0$ such that $\omega_{\tilde{k}}\leq\delta$.

Assume we use (\ref{lb}) to calculate lower bound. According to the way $\mathcal{U}^{p}$ is constructed, for every $u\in\mathcal{U}^{p}$, there exists a lower bound $l\in\mathcal{L}^{p}$, such that
\begin{align*}
d(u,\mathcal{L})\leq d(u,l)= \frac{1}{2}\omega_k\|L\|,
\end{align*}
where $L=(L_1,\ldots,L_m)^T$ consisting of the Lipschitz constants of objectives. Hence we have
\begin{align}
d_h(\mathcal{U}^{p},\mathcal{L}^{p}) = \frac{1}{2}\omega_k\|L\|.\label{eq1}
\end{align}
Due to the fact that $\omega_k$ converges to 0, it follows that for a given $\varepsilon>0$, there must exist a iteration count $\bar{k}>0$ such that $d_h(\mathcal{U}^{p},\mathcal{L}^{p})\leq\varepsilon$.\qed
\end{proof}

By following the same line of reasoning we can find the minimum total number of iterations for RBB to terminate.
\begin{theorem}\label{mti}
Let the predefined parameters $\varepsilon>0$ and $\delta>0$ be given. If $\Omega=\{x\in\mathbb{R}^n:0\leq x_k\leq 1,\;k=0,\ldots,n\}$
%each component of the width of $\Omega$ is equal to 1
, then the minimum total number of iterations for RBB to terminate is
\begin{align*}
  I^{min}=\max\Big\{n\Big\lceil\log_2 \frac{\sqrt{n}}{\varepsilon}\Big\rceil,n\Big\lceil\log_2 \frac{\sqrt{n}\|L\|}{\delta}-1\Big\rceil\Big\}.
\end{align*}
\end{theorem}
\begin{proof}
At the $k$-th iteration, the width of each box $B$ in $\mathcal{B}_k$ is
$$\omega(B)=\big(\underbrace{\frac{1}{2^{K+1}},\ldots,\frac{1}{2^{K+1}}}_{N},\underbrace{\frac{1}{2^{K}},\ldots,\frac{1}{2^{K}}}_{n-N}\big),$$
where $K=\lfloor k/n\rfloor$ and $N=k {\rm mod}(n)$. Then we have
\begin{align}
  \omega_k=\sqrt{N(\frac{1}{2^{K+1}})^2+(n-N)(\frac{1}{2^{K}})^2}=\sqrt{(n-\frac{3}{4}N)(\frac{1}{2^{K}})^2}\leq\frac{\sqrt{n}}{2^{K}}.\label{ineq1}
\end{align}
According to inequality (\ref{ineq1}), for a given $\delta>0$, we know $\omega_k\leq\delta$ if
\begin{align}
  k= n\Big\lceil\log_2 \frac{\sqrt{n}}{\delta}\Big\rceil.\label{eq2}
\end{align}

On the other hand, substituting (\ref{ineq1}) into (\ref{eq1}), it follows that
\begin{align*}
d_h(\mathcal{U}^{p},\mathcal{L}^{p}) = \frac{1}{2}\omega_k\|L\|\leq\frac{\sqrt{n}}{2^{K+1}}\|L\|,
\end{align*}
thus, for a given $\varepsilon>0$, we know $d_h(\mathcal{U}^{p},\mathcal{L}^{p})\leq\varepsilon$ if
\begin{align}
  k= n\Big\lceil\log_2 \frac{\sqrt{n}\|L\|}{\delta}-1\Big\rceil.\label{eq3}
\end{align}

By (\ref{eq2}) and (\ref{eq3}), we complete the proof. \qed
\end{proof}

Observe that the condition in Theorem \ref{mti} requiring the width of each side of the feasible region to be equal to 1 is mild, since we could normalize sides of the original feasible region, meanwhile, the Lipschitz constants of the objectives will also change. In practice, the total number of iterations is smaller than $I^{min}$ due to the inequality (\ref{ineq1}) and the update of the Lipschitz constants during the iterations.

In the following  we want to show $\mathcal{X}$ output by RBB is a set of $\varepsilon$-efficient solutions of problem (\ref{MOP}).

\begin{theorem}
Let $\mathcal{X}$ be the preferred solution set generated by RBB and $\mathcal{L}^p$ the nondominated lower bound set. Then $\tilde{x}\in\mathcal{X}$ is an $\varepsilon$-efficient solution of problem (\ref{MOP}).

\end{theorem}
\begin{proof}
  Suppose $\tilde{x}$ is the midpoint of the box $B\in\mathcal{B}_k$ and $l\in\mathcal{L}^p$ is corresponding lower bound. According to (\ref{eq1}), we know that
  \begin{align}
    \varepsilon\geq \frac{1}{2}\omega_k\|L\|>\frac{1}{2}\omega_k L_{{\rm max}},\label{ineq2}
  \end{align}
  where $L_{max}=\max\{L_i,i=1,\dots,m\}$. Then we can obtain a lower bound $\tilde{l}=(\tilde{l}_1,\ldots,\tilde{l}_m)^T$ whose component can be calculated by
  \begin{align*}
    \tilde{l}_i=f(\tilde{x})_i-\frac{1}{2}\omega_k L_{{\rm max}},
  \end{align*}
  and further, it is easy to see that $\tilde{l}\leqq l$.

  On the one hand, for every $x\in B$, we have
  \begin{align*}
    F(\tilde{x})-\varepsilon e <\tilde{l}\leqq l\leq F(x).
  \end{align*}

  On the other hand, assume there exists another box $B'$ and there exists a feasible point $x'\in B'\in\mathcal{B}_k\backslash B$ such that $F(x')\leq F(\tilde{x})-\varepsilon e$.

  By (\ref{lb}), (\ref{ineq2}) and the Pareto dominance relation between $l$ and $\tilde{l}$, we have
    \begin{align*}
    F(x')-\frac{1}{2}\omega_k\|L\|\leq F(x')\leq F(\tilde{x})-\varepsilon e<\tilde{l}\leqq l,
  \end{align*}
  which is a contradiction to the fact $\mathcal{L}^p$ is a nondominated lower bound set. Thus, $\tilde{x}\in\mathcal{X}$ is an $\varepsilon$-efficient solution of problem (\ref{MOP}). \qed
\end{proof}
\section{Experimental Results}
RBB has been implemented in Python 3.8 with fundamental packages like numpy, scipy and multiprocessing. Now we show the experimental results on 2 to 7 objectives using the RBB. In all experiments, we use MOEA/D-DE \cite{ref24} with the population size 10 and 20 generations in the proposed algorithm.

\subsection{Test problem ZDT1}
First, we consider the 10-variable ZDT1 problem with $x=(x_1,\ldots,x_n)\in[0.2, 1]\times[0, 1]^{n-1}$. The Pareto front spans continuously in $f_1\in[0.2,1]$ and follows a function relationship $f_2=1-\sqrt{f_1}$. To investigate the effect of a weight vector in the distribution of preferred solutions, we use RBB with $\sigma=0.005$ and three different weight vectors: $(0.5, 0.5)$, $(0.8, 0.2)$ and $(0.2, 0.8)$ on ZDT1 problem. A reference point $r=(0.4, 0.15)$ and the precision parameters $(\varepsilon,\delta)=(0.0015, 0.00015)$ are used. Fig. 1(a) shows the influence of the weight vectors on the distribution of objective vectors. As expected, the obtained objective vectors with the first weight vector do not put emphasis on either of objectives. For the second weight vector, more emphasis is placed on $f_1$, thus the obtained objective vectors are closer to the minimum of $f_1$. On the contrary, the solutions with the third weight vector are closer to the minimum of $f_2$. The results show that if the decision maker is interest in some objectives more than the others, a biased distribution of objective vectors can be obtained by RBB based on the reference point. In the subsequent numerical experiments, we use a uniform weight vector.

\subsection{Test problem ZDT2}
The 10-variable ZDT2 problem is considered next. The value of the Pareto front satisfies $f_2=1-f_1^2$ with $f_1\geq 0$. Fig. 1(b) shows the effect of different $\sigma$ values on the range of preferred solutions. A reference points $r=(0.3,0.3)$ and the precision parameters $(\varepsilon,\delta)=(0.0015, 0.00015)$ are chosen for this problem. Three different $\sigma$ values of 0.005, 0.015 and 0.05 are chosen. Objective vectors with $\sigma=0.005$ are shown on the ture Pareto front. Objective vectors with other $\sigma$ values are shown with an offset to the true Pareto front. It is clear that the larger $\sigma$ value is, the larger the range of objective vectors obtained. Therefore, if the decision maker want to obtain a large range of objective vectors, a large value of $\sigma$ can be chosen.

\subsection{Test problem ZDT3}
The 5-variable ZDT3 problem has a disconnected set of Pareto fronts. We consider four reference points, of which three are infeasible and one is feasible. Fig. 1(c) shows the obtained objective vectors with $\sigma=0.003$ and $(\varepsilon,\delta)=(0.005, 0.0002)$.  It can be clearly seen that disconnection of the Pareto front and the feasibility of reference points do not cause any difficulty to the proposed algorithm.

\begin{figure}[H]
  \centering
  \subfigure[]{
    \includegraphics[width=0.32\textwidth]{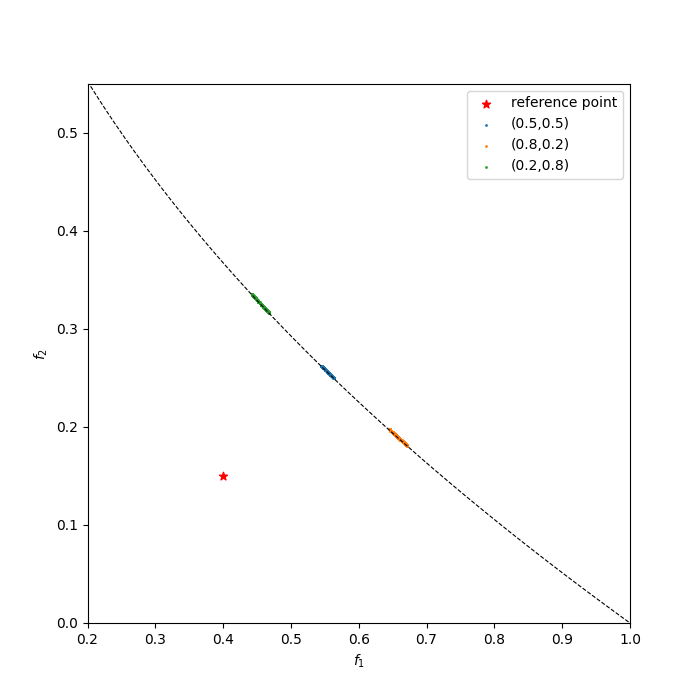}}
  \subfigure[]{
    \includegraphics[width=0.32\textwidth]{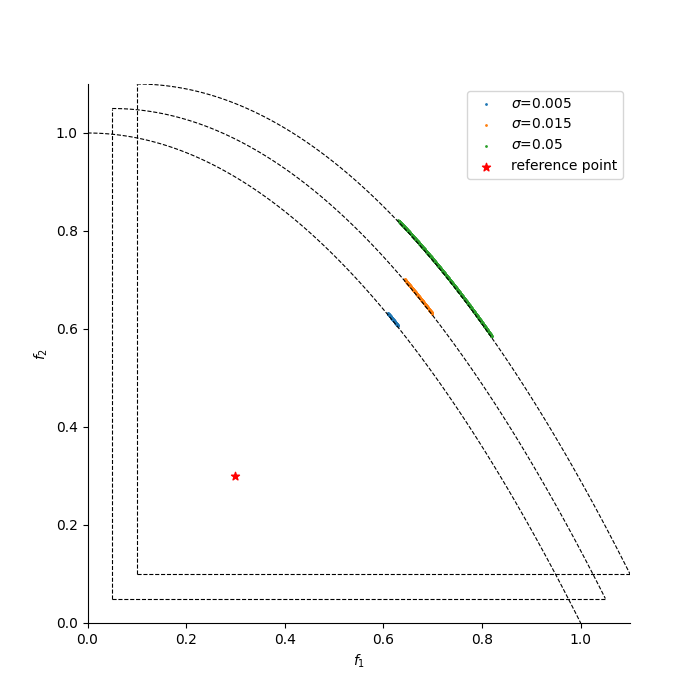}}
  \subfigure[]{
    \includegraphics[width=0.32\textwidth]{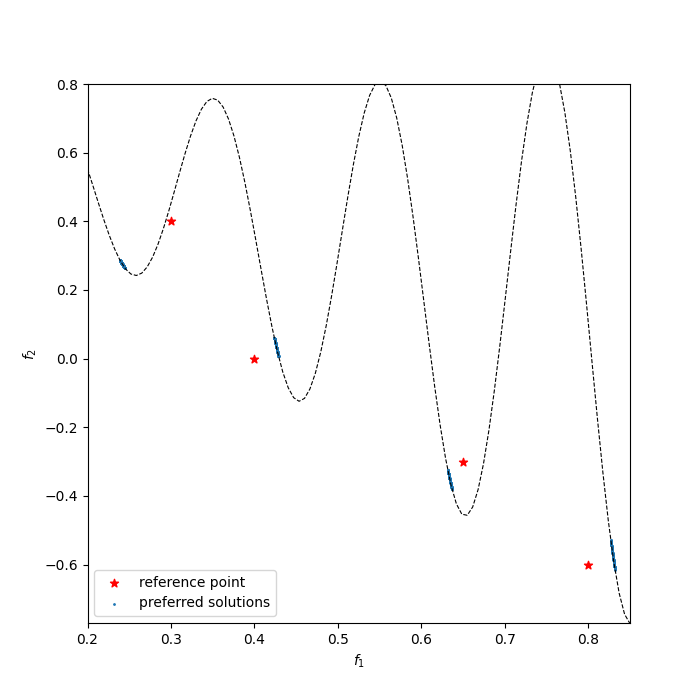}}

  \caption{Objective vectors obtained for the ZDT problem. \textbf{a} Biased objective vectors with different weight vectors according to a reference point for ZDT1. \textbf{b} Effect of $\sigma$ in obtaining distribution of objective vectors on ZDT2. \textbf{c} Objective vectors obtained for four reference points on ZDT3. }
\end{figure}

\subsection{Three-objective DTLZ2 problem}
The 3-objective, 7-variable DTLZ2 problem has a non-convex Pareto front whose function value satisfies $\sum_{i=1}^3 f_i^2=1$. Two reference points are chosen as follows: (i) (0.4, 0.4, 0.8) and (ii) (0.8, 0.8, 0.6). We use $\sigma=0.004$ here. Fig. 2(a) shows the obtained objective vectors with the precision parameters $(\varepsilon,\delta)=(0.004,0.006)$. This demonstrates the applicability of the proposed algorithm in solving three-objective optimization problems.

\subsection{Five-objective DTLZ2 problem}
Here, we apply $\sigma=0.005$ and $(\varepsilon,\delta)=(0.006,0.008)$ to the 5-variable DTLZ2 problem. Two reference points are chosen as follows: (i) (0.8, 0.8, 0.8, 0.8, 0.8) and (ii) (0.2, 0.2, 0.2, 0.2, 0.8). Fig. 2(b) shows the value-path plot of the obtained objective vectors. It is obvious that two distinct sets of objective vectors according to the above reference points are obtained by the solution process of RBB. Since the Pareto front of the DTLZ2 problem satisfies $\sum_{i=1}^{m}f_i^2=1$, we compute the left side of this expression for all obtained solutions and the values are found to lie within $[1,1.0399]$ (at most 3.99\% from one), which means that all objective vectors are very close to the true Pareto front.
\subsection{Seven-objective DTLZ2 problem}
We then attempt to solve 7-objective, 7-variable DTLZ2 problem with one reference point: $f_i=0.25$ for all $i = 1, 2,..., 10.$
We use $\sigma=0.002$ and $(\varepsilon, \delta)=(0.006, 0.008)$. The value-path plot of the objective vectors is shown in Fig. 2(c). It is clear that the reference point effectively guides the search towards the region of interest, so that the images of the obtained objective vectors concentrates near $f_i=0.4$. When we compute $\sum_{i=1}^{m}f_i^2$ of all obtained solutions,
the values are found to lie within $[1,1.0076]$ (at most 0.76\% from one), thereby meaning that all objective vectors are almost on the true Pareto front.

\begin{figure}[H]
  \centering
      \subfigure[]{
    \includegraphics[width=0.32\textwidth]{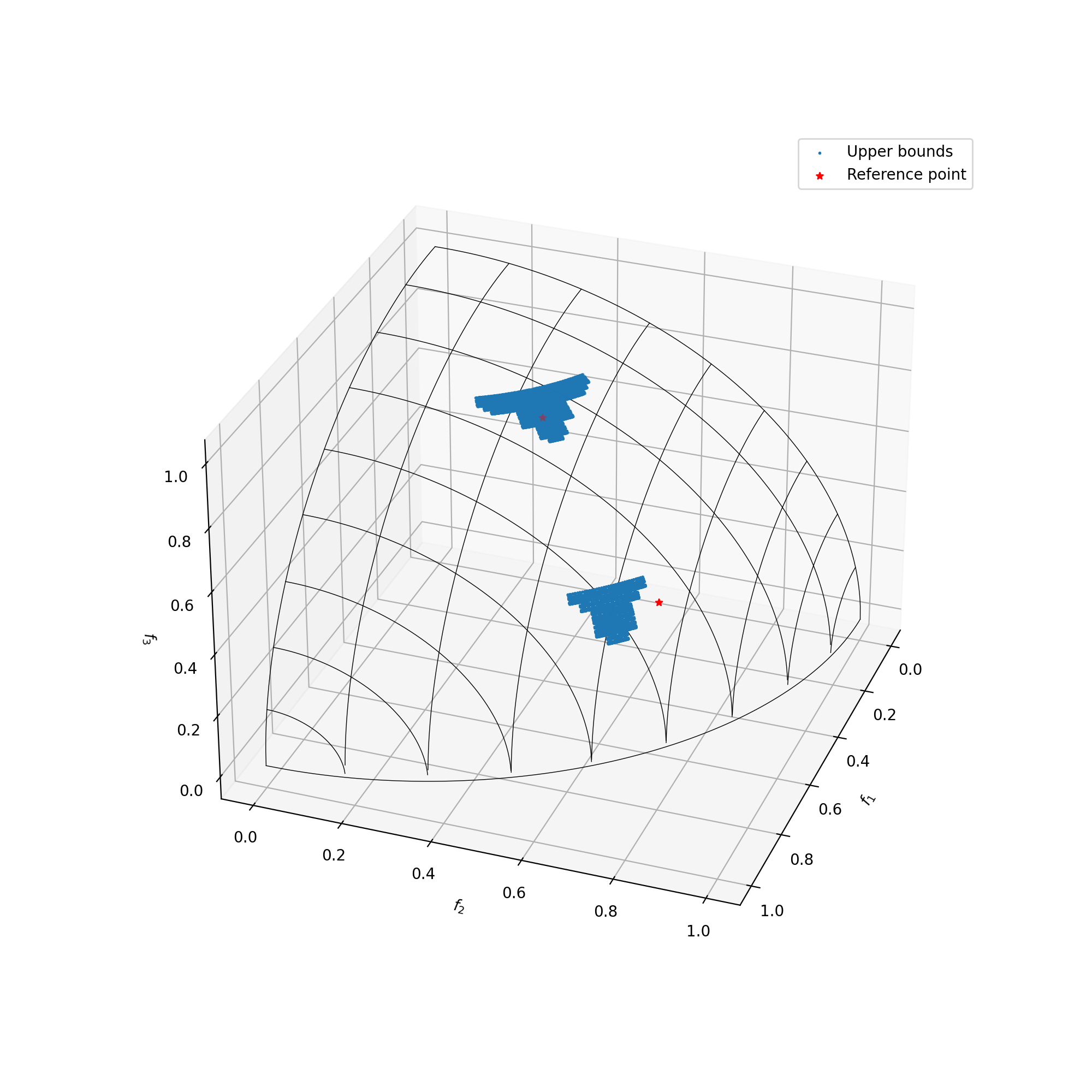}}
  \subfigure[]{
    \includegraphics[width=0.32\textwidth]{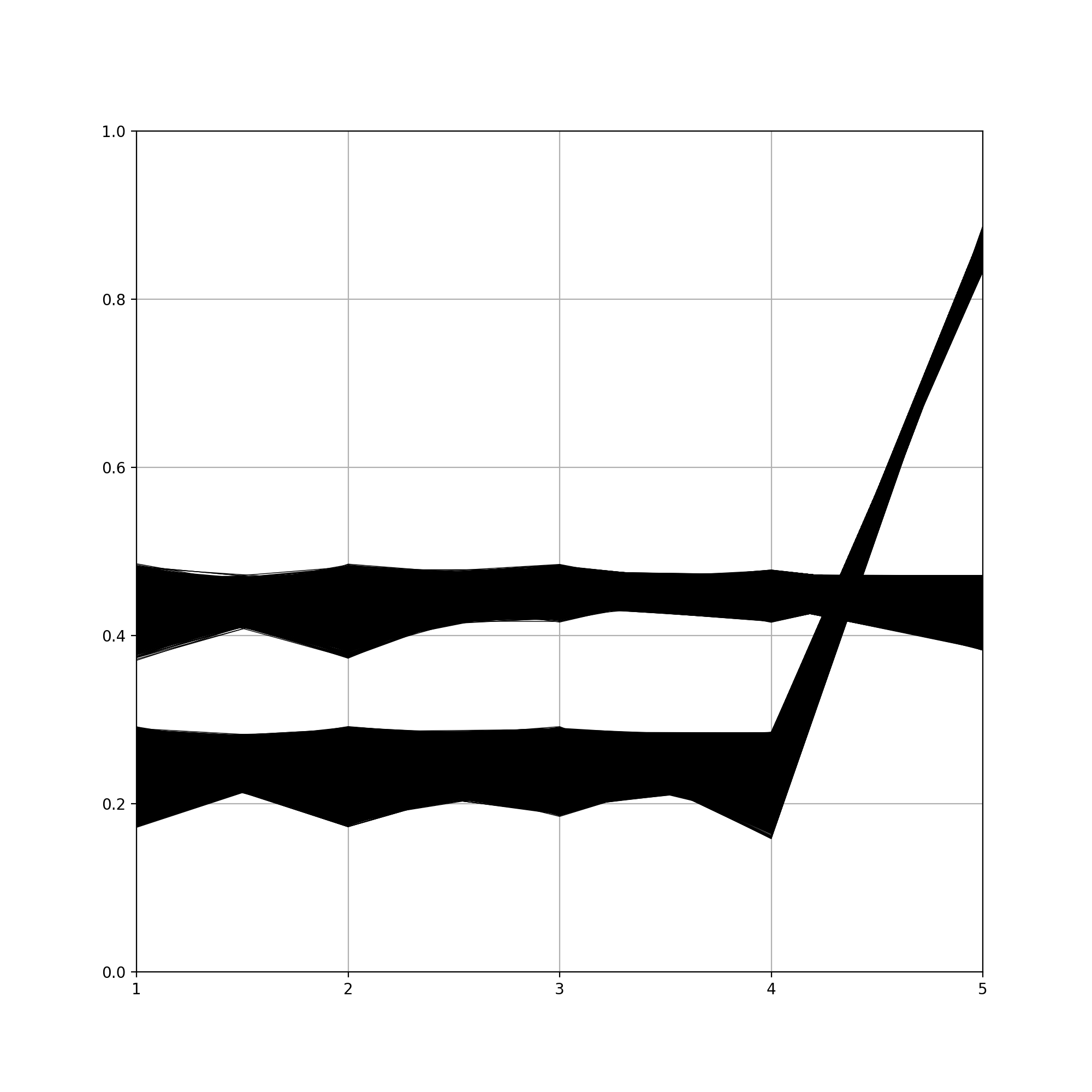}}
  \subfigure[]{
    \includegraphics[width=0.32\textwidth]{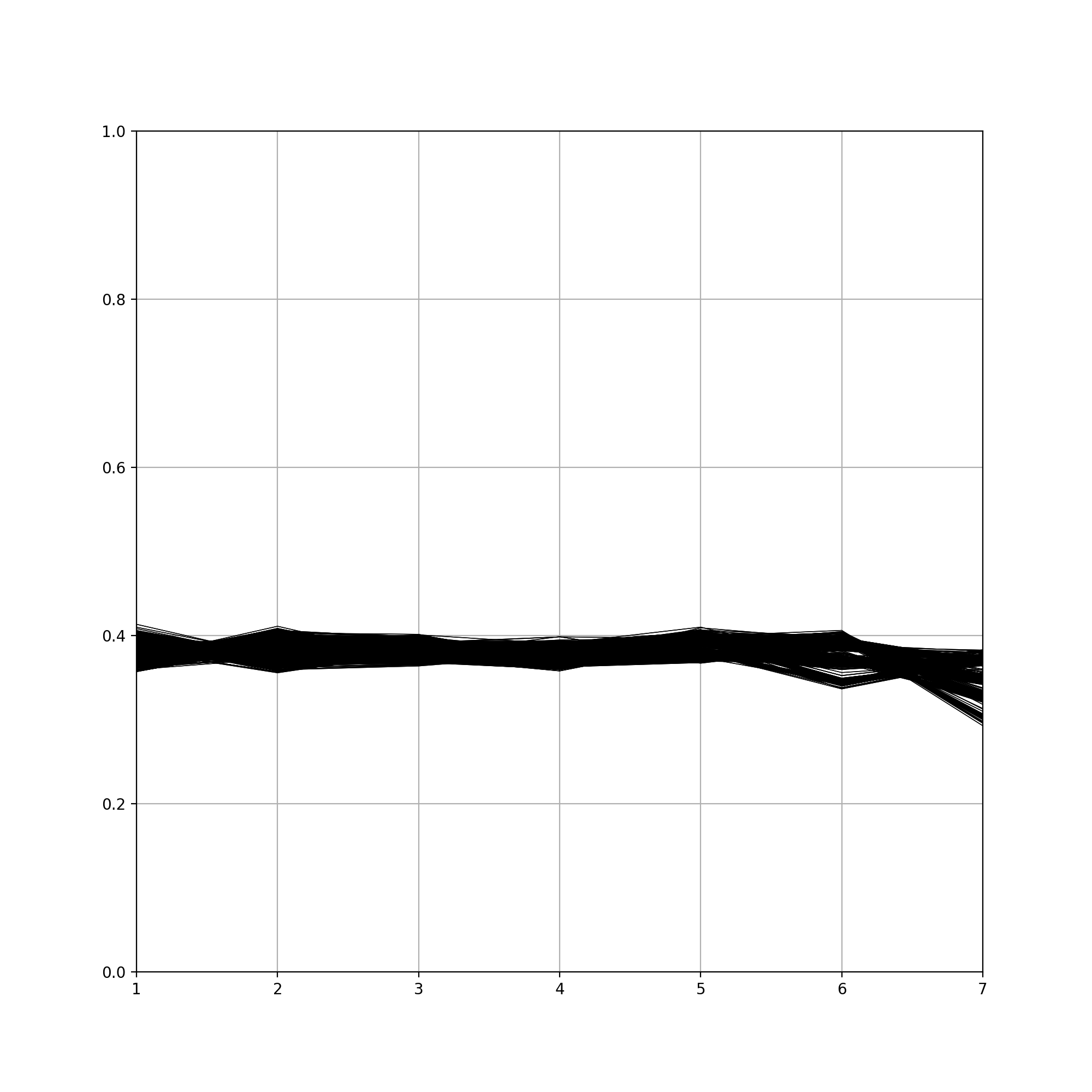}}
  \caption{Objective vectors obtained by RBB on the DTLZ2 test problem. \textbf{a} Objective vectors obtained for two reference points on 3-objective DTLZ2. \textbf{b} Objective vectors obtained for two reference points on 5-objective DTLZ2. \textbf{c} Objective vectors obtained for a reference point on 7-objective DTLZ2.}
\end{figure}

\subsection{Comparison experiment}
We compare RBB with g-NSGA-II \cite{ref15} and WASF-GA \cite{ref19} in the ZDT and DTLZ test problems. Due to lack of space, we will plot just the objective vectors generated in the ZDT2 problem. Both of g-NSGA-II and WASF-GA have a population size of 100 individuals and 250 generations for this problem. The $\sigma$ value of 0.015 for RBB is chosen. The objective vectors provided by each algorithm and the reference points can be seen in Fig. 3. It is easy to see that the range of objective vectors obtained by RBB is quite different from those obtained by g-NSGA-II and WASF-GA. This is because, in g-NSGA-II and WASF-GA, the region of interest is defined by the Pareto dominance relation between the reference point and objective vectors. Therefore, the distance between the reference point and the Pareto front significantly determines the range of the region of interest. As a result, if a reference point is provided that is far from the Pareto front, it may still be difficult for the decision maker to find a preferred solution among the alternatives. In contrast, the range of the region of interest generated by RBB is only associated with the $\sigma$ value, so the range of the region of interest is fixed whether the reference point is far from the Pareto front surface or not.

\begin{figure}[H]
  \centering
  \subfigure[]{
    \includegraphics[width=0.31\textwidth]{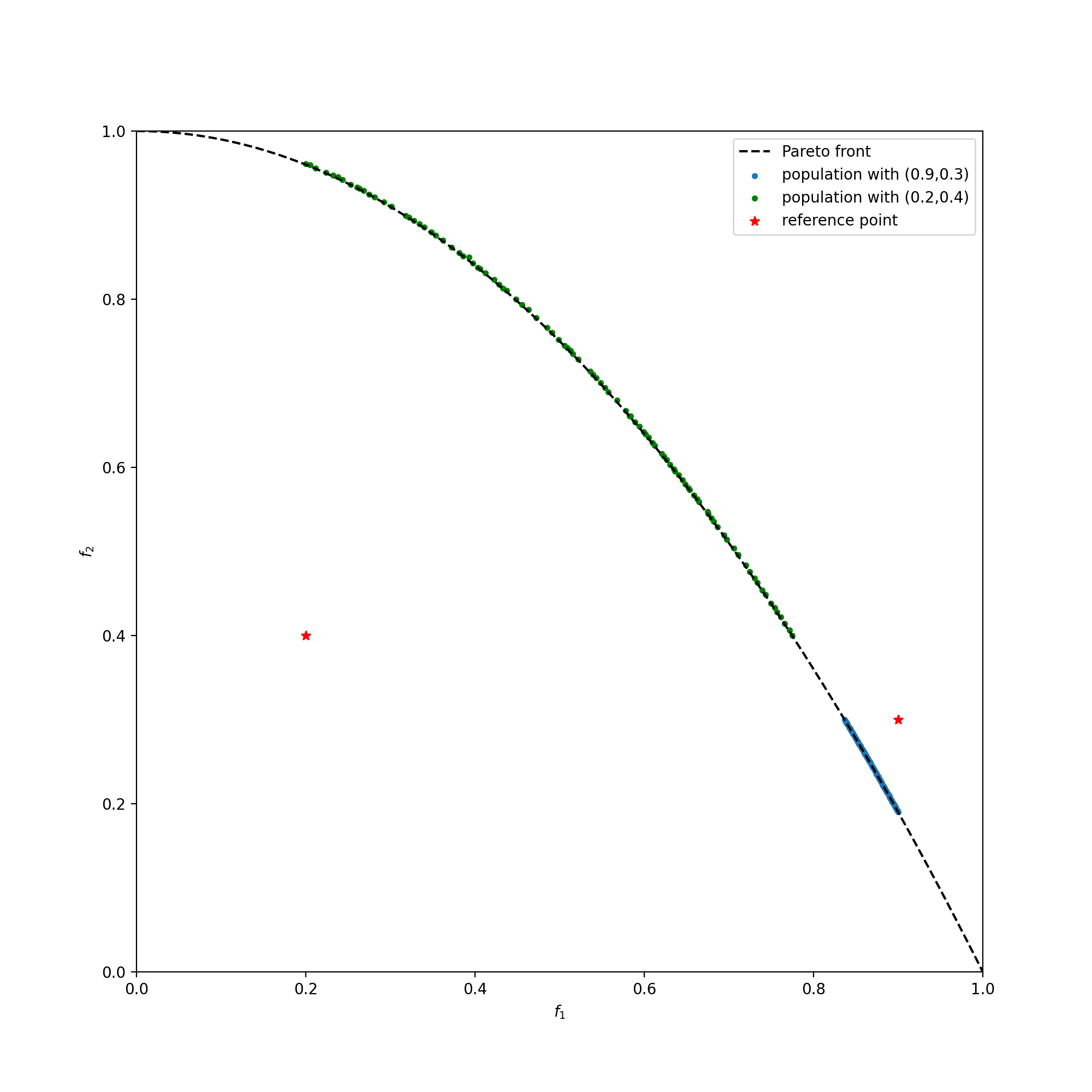}}
  \subfigure[]{
    \includegraphics[width=0.31\textwidth]{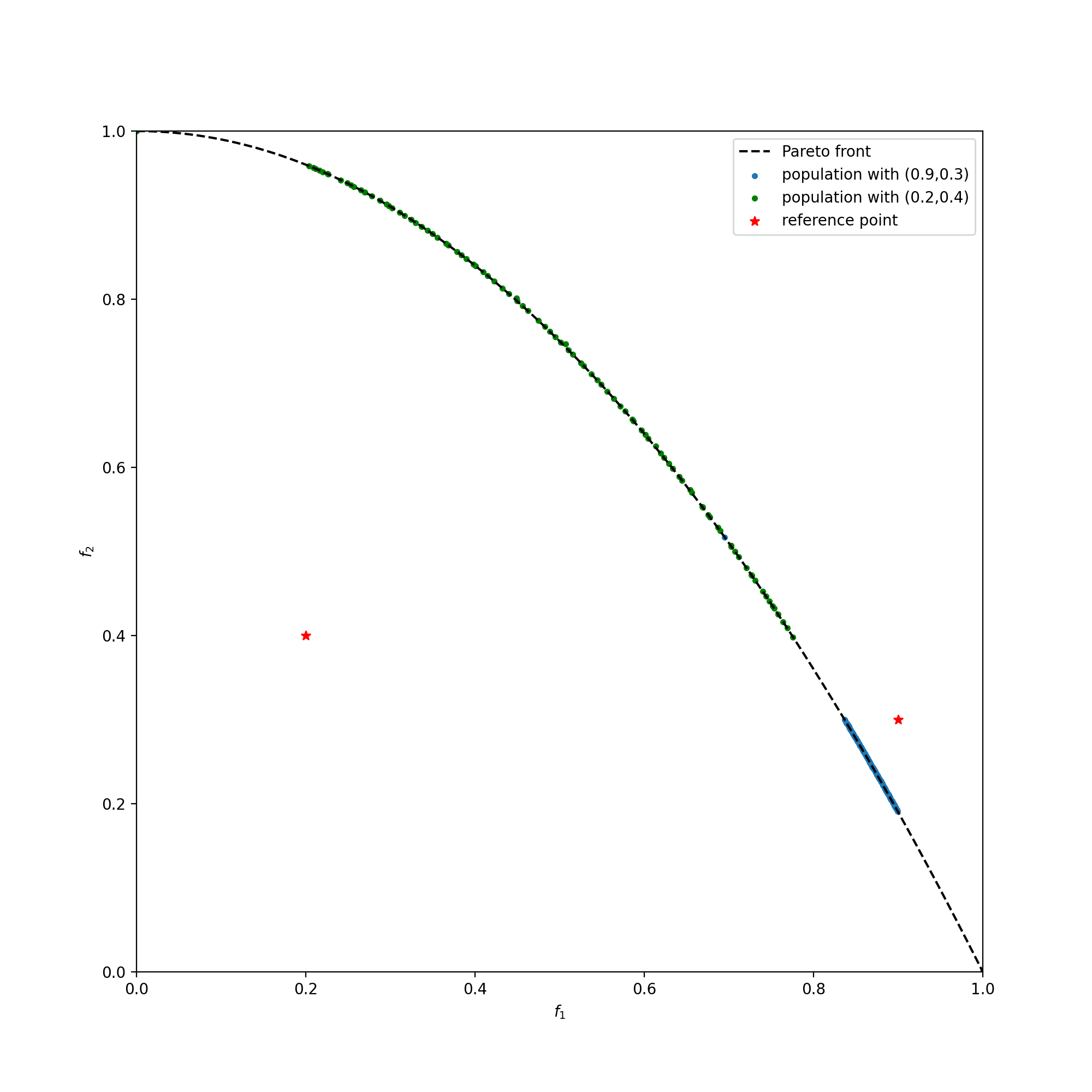}}
      \subfigure[]{
    \includegraphics[width=0.31\textwidth]{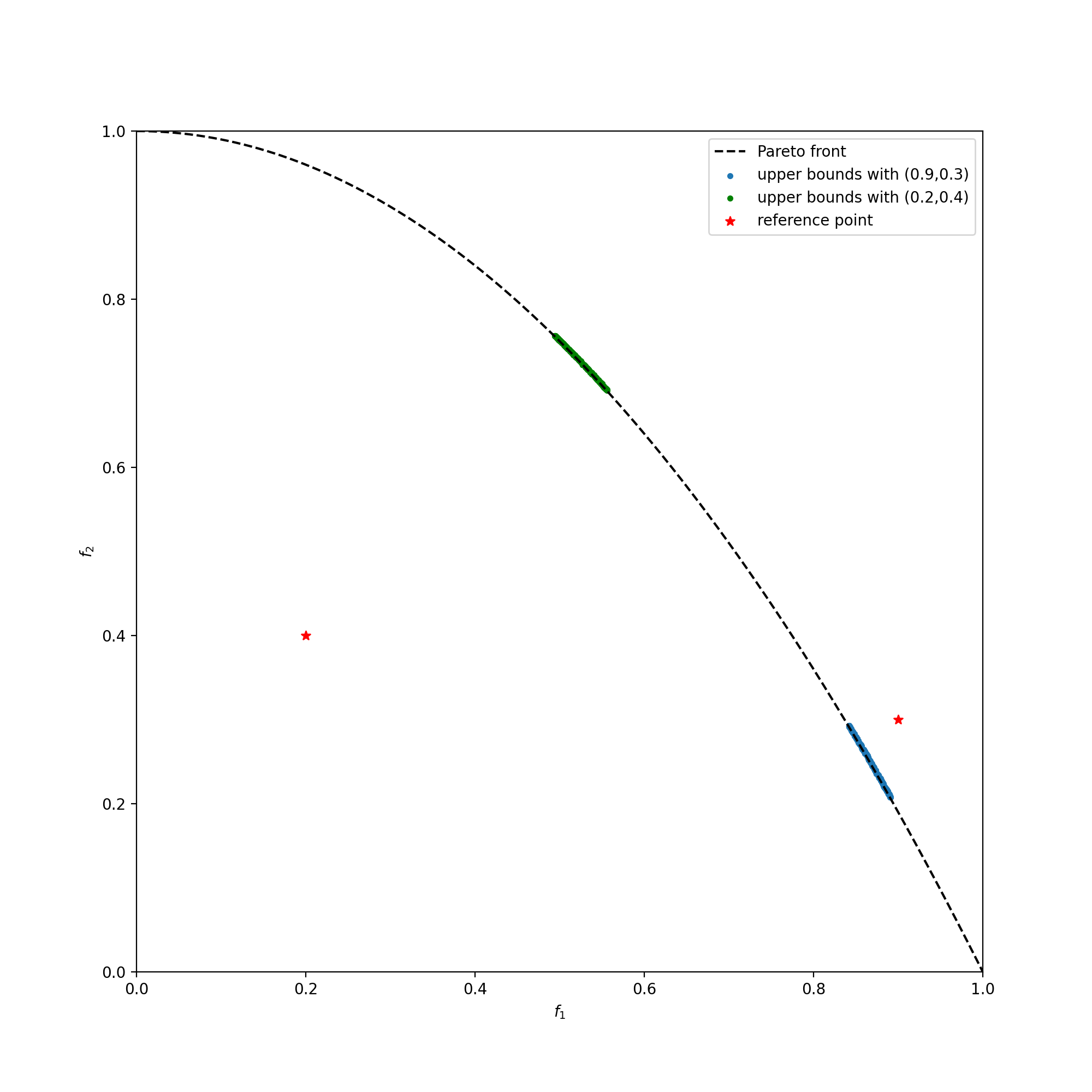}}
  \caption{Objective vectors obtained for the ZDT2 test problem. \textbf{a} Objective vectors obtained by g-NSGA-II. \textbf{b} Objective vectors obtained by WASF-GA. \textbf{c} Objective vectors obtained by RBB.}
\end{figure}

Besides, many practical problems are less methodically constructed than the ZDT or DTLZ problems, and thus tend to be more difficult to analyse. Of particular interest is that the following VNT problem \cite{ref21} are nonseparable and multimodal, characteristics that are known to be more representative of practical problems,
$$F(x)=
\begin{pmatrix}
0.5(x_1^2+x_2^2)+\sin(x_1^2+x_2^2)\\
\frac{(3x_1-2x_2+4)^2}{8}+\frac{(x_1-x_2+1)^2}{27}+15\\
\frac{1}{x_1^2+x_2^2+1}-1.1\exp(-x_1^2-x_2^2)
\end{pmatrix},(x_1,x_2)\in[-3,3]^2.$$

Here, we compare the objective vectors retrieved by RBB with those generated by g-NSGA-II and WASF-GA on the VNT problem. For RBB, the precision parameters $(\varepsilon,\delta)=(0.03, 0.007)$ and the region parameter $\sigma=0.0003$ are chosen. A reference point $r$ is $(4, 10, 0)$. Fig. 4 shows the regions of interest generated by each of the three algorithms respectively. It is easy to see that g-NSGA-II is unable to identify the correct region of interest. Although WASF-GA finds several Pareto solutions on the region of interest, these solutions do not perfectly represent the region of interest corresponding to the reference point. The main reason for the poor performance of g-NSGA-II and WASF-GA is that they determine the region of interest by the Pareto dominance, which makes it easier to handle the simplex-like front. However, the Pareto front of this test problem consists of a degenerate convex line and a degenerate mixed convex/concave line. This means that these degenerate lines occupy only a small part of the search region, so most of the search is wasted in two algorithms. Compared to the previous two algorithms, RBB generates a set of Pareto solutions distributed in the region of interest corresponding to the reference point.%Given that the three algorithms are designed based on different methodologies, they yield different regions of interest, thus we only investigate the performance of the three algorithms on this test problem.

\begin{figure}
  \centering
  \subfigure[]{
    \includegraphics[width=0.31\textwidth]{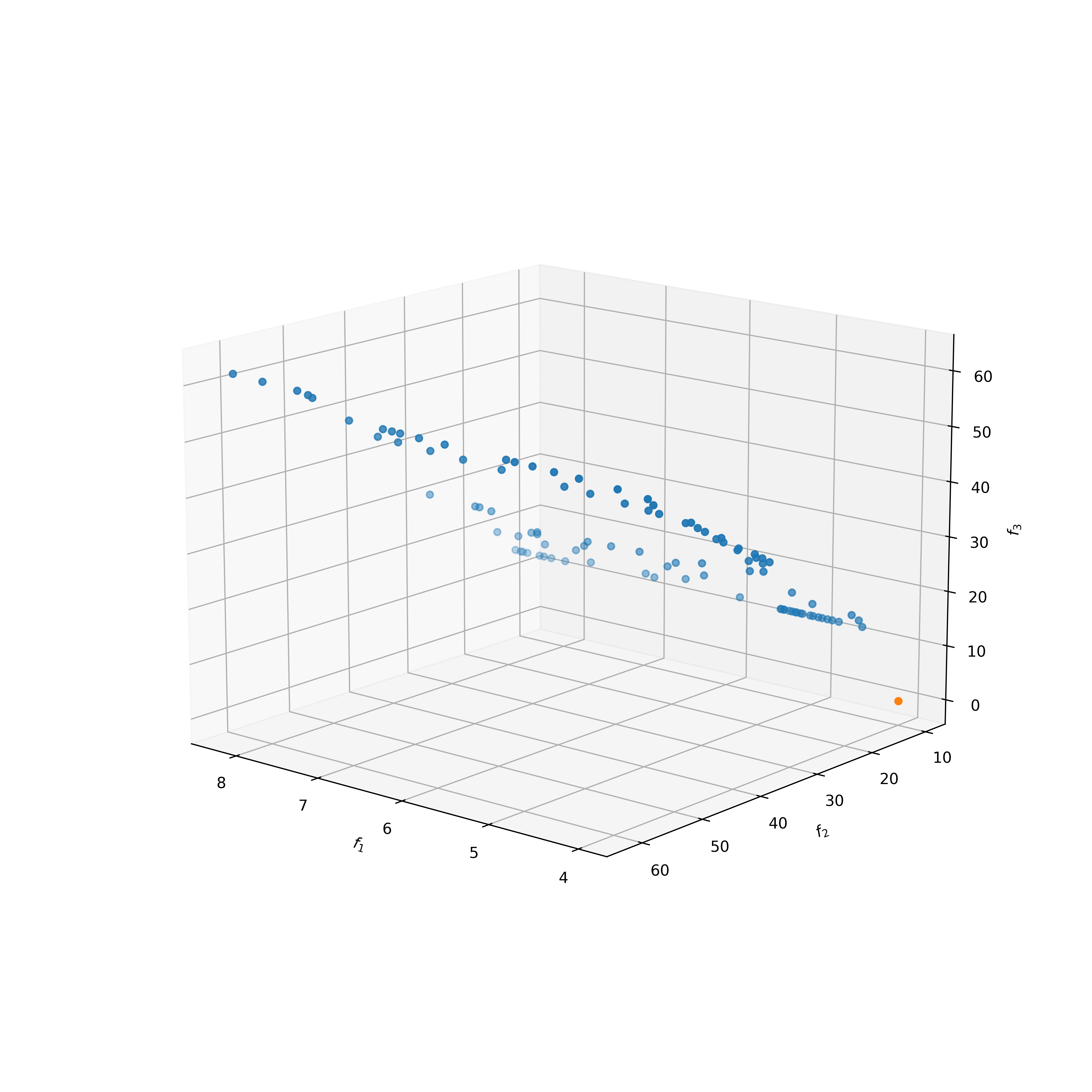}}
  \subfigure[]{
    \includegraphics[width=0.31\textwidth]{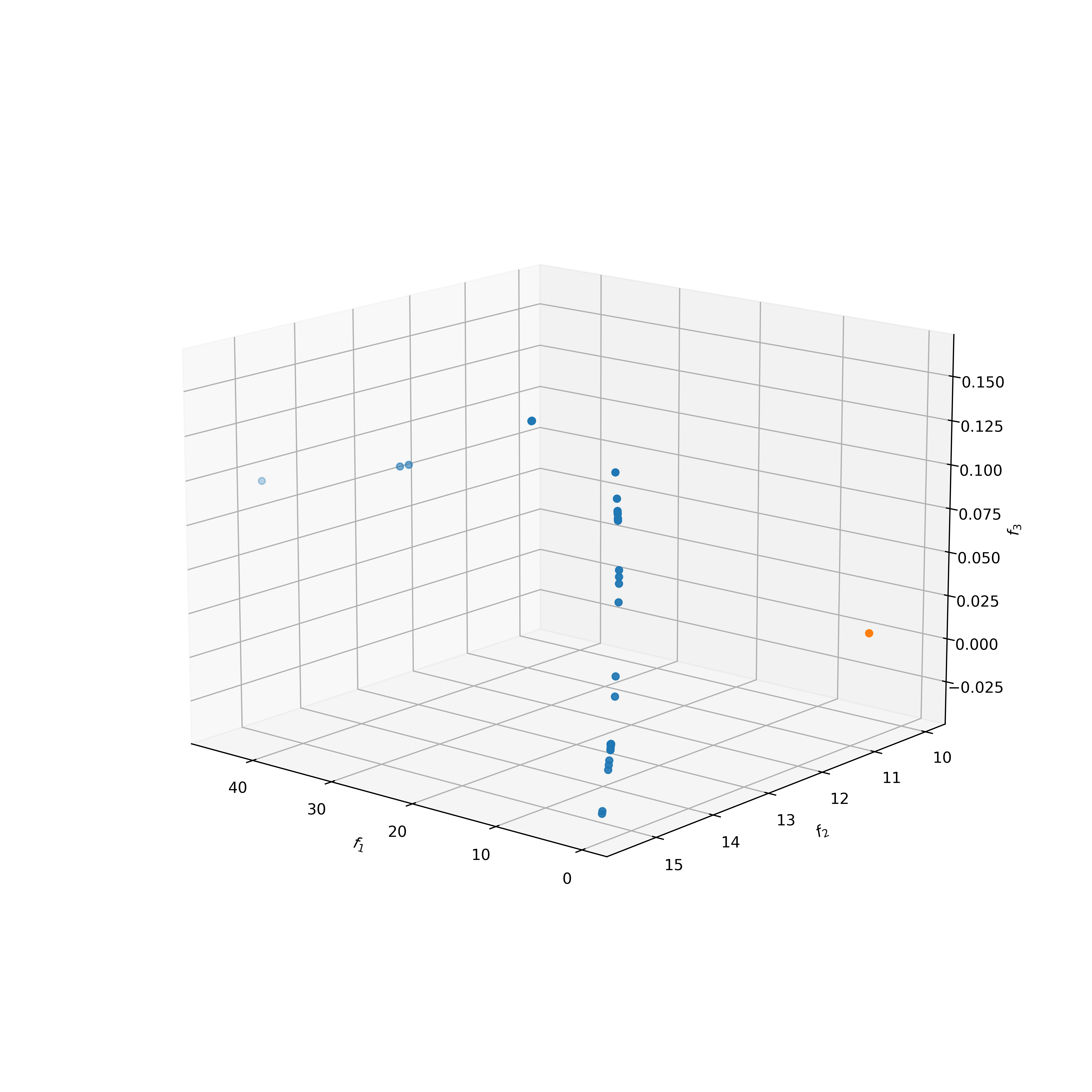}}
      \subfigure[]{
    \includegraphics[width=0.31\textwidth]{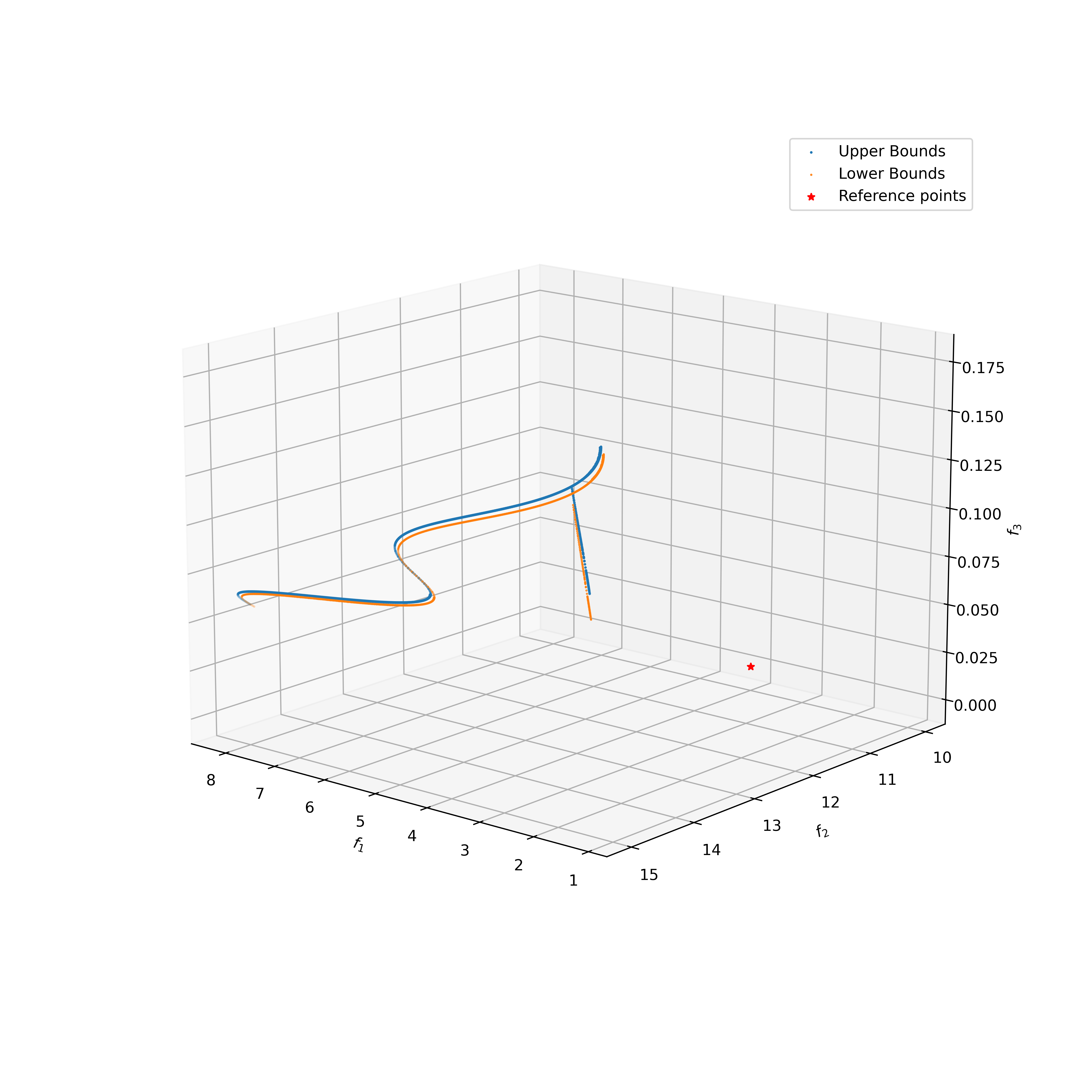}}
  \caption{Objective vectors obtained for the VNT test problem. \textbf{a} Objective vectors obtained by g-NSGA-II. \textbf{b} Objective vectors obtained by WASF-GA. \textbf{c} Objective vectors obtained by RBB.}
\end{figure}

\subsection{Welded beam design problem}
The welded beam design problem \cite{ref3,ref4,ref18} has four real-parameter variables $x=(x_1, x_2, x_3, x_4)$ and four non-linear constraints. One of the two objectives is to minimize the cost of fabrication and other is to minimize the end deflection of the welded beam:
$$F(x)=
\begin{pmatrix}
1.10471x_1^2 x_3+ 0.04811x_2x_4(14.0+x_3)\\
2.1592/(x_2x_4^3)
\end{pmatrix}$$
subject to the constraints
\begin{align*}
&g_1(x)=13600-\tau(x)\geq0,\\
&g_2(x)=30000-\sigma(x)\geq0,\\
&g_3(x)=x_2-x_1\geq0,\\
&g_4(x)=P_c(x)-6000\geq0,\\
&0.125\leq x_1,x_2\leq 5,\\
&0.1\leq x_3,x_4\leq 10.
\end{align*}
where the stress and buckling terms are non-linear to design variables and are given as follows
\begin{align*}
&\tau(x)=\sqrt{(\tau')^2+(\tau'')^2+(x_3\tau'\tau'')/\sqrt{0.25(x_3^2+(x_1+x_4)^2)}},\\
&\tau'=\frac{6000}{\sqrt{2}x_1x_3},\\
&\tau''=\frac{6000(14+0.5x_3)\sqrt{0.25(x_3^2+(x_1+x_4)^2)}}{1.414x_1x_3(x_3^2/12+0.25(x_1+x_4)^2)},\\
&\sigma(x)=\frac{504000}{x_2x_4^2},\\
&P_c(x)=64746.022(1-0.0282346x_4)x_4x_2^3.
\end{align*}

Here, instead of finding the complete Pareto front, we are interested in finding the regions corresponding to three chosen reference points: (i) (4, 0.003), (ii) (20, 0.002), (iii) (32, 0.0007). Figs 4(a) shows the results with RBB with $\sigma=0.0001$ and $(\varepsilon, \delta)=(0.3, 0.02)$ on the welded beam design problem. It is easy to see that if the decision maker is interested in knowing trade-offs in three major areas (minimum cost, intermediate between cost and deflection, and minimum deflection), RBB is able to characterize the regions of interest, instead of finding the whole Pareto front, thus allowing the decision maker to consider only a few solutions distributing among the regions of interest. Furthermore, when the decision maker provides a feasible reference point such as the second reference point, meaning that his/her desirable aspiration levels are conservative. In this case, RBB can provide a set of $\varepsilon$-efficient solutions which are better than the given reference point.

\begin{figure}[H]
  \centering
  \subfigure[]{
    \includegraphics[width=0.45\textwidth]{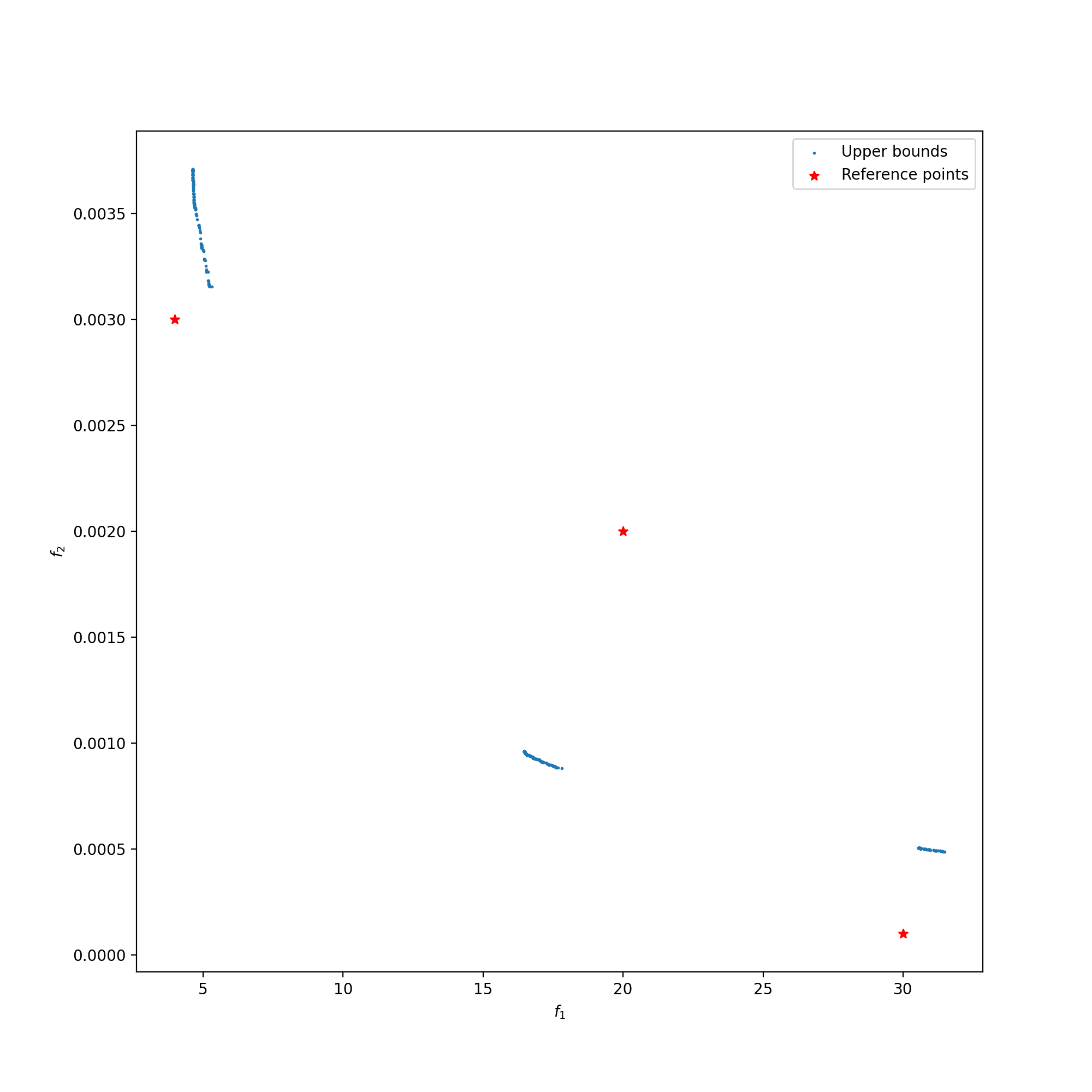}}
  \subfigure[]{
    \includegraphics[width=0.45\textwidth]{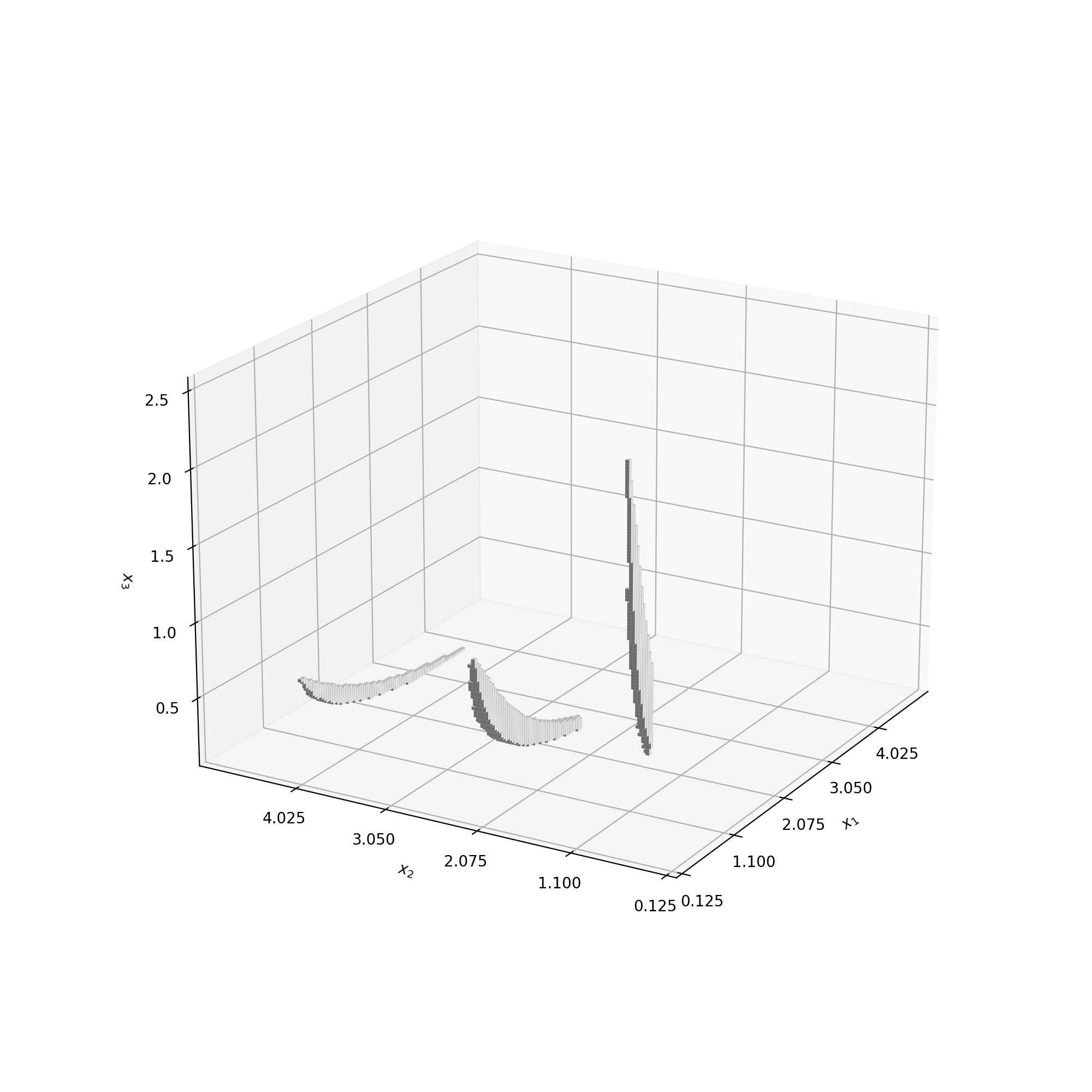}}
  \caption{Computational results obtained on the welded beam design problem. \textbf{a} Objective vectors obtained by RBB for three reference points. \textbf{b} Projections of Subboxes containing the preferred solutions onto the $x_1-x_2-x_3$ space.}
\end{figure}

\section{Conclusion}
Most of the current branch and bound algorithms for MOPs are presented to approximate the complete Pareto set, but different parts of the Pareto set might be more preferred by the decision maker than some others, while some parts might not be of interest at all. Therefore, not only are computational resources wasted in exploring undesired regions, but also the cognitive pressure of decision maker increases.

To avoid this shortcoming, we have proposed a new multiobjective optimization method called \emph{reference-point-based branch and bound algorithm} (RBB), which tries to approximate the region of interest of the Pareto front corresponding to the reference point. To achieve this purpose, RBB employs a new discarding test which takes into account the values of each lower and upper bounds on an ASF, and further controls the discarding pressure by a predefined parameter $\sigma$. In addition, RBB uses the heuristic search to improve the solution quality. We have proven that the candidate solutions obtained by RBB are $\varepsilon$-efficient solutions whose images are distributed in the region of interest. The usefulness of RBB has been demonstrated on several test problems including the ZDT test problems, the 3-to 7-objective DTLZ2 test problems and the welded beam design problem.

\begin{acknowledgements}
This work is supported by the Major Program of National Natural Science Foundation of China (Nos. 11991020, 11991024), the General Program of National Natural Science Foundation of China (No. 11971084), the Team Project of Innovation Leading Talent in Chongqing (No. CQYC20210309536), the Funds for International Cooperation and Exchange of the National Natural Science Foundation of China (No. 12261160365) and the Natural Science Foundation of Chongqing (No. cstc2019jcyj-zdxmX0016)
\end{acknowledgements}

\end{document}